\documentclass{amsart}

\usepackage{amssymb, enumerate}

\usepackage[all]{xy}

\newtheorem{thm}{Theorem}
\newtheorem*{thm*}{Theorem}

\newtheorem{lem}[thm]{Lemma}
\newtheorem{cor}[thm]{Corollary}

\newtheorem{prop}[thm]{Proposition}

\newtheorem{conj}[thm]{Conjecture}
\newtheorem*{conj*}{Conjecture}
   
\theoremstyle{definition}

\newtheorem{exmp}[thm]{Example}

\newtheorem*{ques}{Question}    

\newtheorem{rem}[thm]{Remark}          

\newtheorem*{ack}{Acknowledgments}      
\newtheorem{notation}[thm]{Notation}   
  
\newtheorem{defn-thm}[thm]{Definition--Theorem}  
\newtheorem{defn-lem}[thm]{Definition--Lemma}  
\newtheorem{defn-prop}[thm]{Definition--Proposition}

\theoremstyle{remark}

\renewcommand{\o}[0]{{\mathcal O}} 
\newcommand{\z}[0]{{\mathbb Z}}

\newcommand{\p}[0]{{\mathbb P}}

\newcommand{\pic}[0]{\operatorname{Pic}}
\newcommand{\num}[0]{\operatorname{Num}}

\newcommand{\sF}{\mathcal{F}}

\def\loccoh#1.#2.#3.#4.{H^{#1}_{#2}(#3,#4)}

\DeclareMathAlphabet{\mathchanc}{OT1}{pzc}%
                                {m}{it}

\begin{document}
\bibliographystyle{amsalpha}

\title[On the Borisov-Nuer conjecture]{On the Borisov-Nuer conjecture and the image of the Enriques--to--$K3$ map}
\author{Marian Aprodu}
\address{Faculty of Mathematics and Computer Science, University of Bucharest, 
\newline
\indent
14 Academiei Street, 010014 Bucharest, Romania}
\email{marian.aprodu@fmi.unibuc.ro}
\address{``Simion Stoilow'' Institute of Mathematics of the Romanian Academy,
\newline
\indent
P.O. Box 1-764, 014700 Bucharest, Romania
}
\email{marian.aprodu@imar.ro}

\author{Yeongrak Kim}
\address{Faculty of Mathematics and Informatics, University of Saarland,
\newline
\indent
Geb. E2.4, 66123 Saarbr{\"u}cken, Germany
}
\email{kim@math.uni-sb.de}

\begin{abstract}
We discuss the Borisov-Nuer conjecture in connection with the canonical maps from the moduli spaces  $\mathcal M_{En,h}^a$of polarized Enriques surfaces with fixed polarization type $h$ to the moduli space $\mathcal F_g$ of polarized $K3$ surfaces of genus $g$ with $g=h^2+1$, and we exhibit a naturally defined locus $\Sigma_g\subset\mathcal F_g$. One direct consequence of the Borisov-Nuer conjecture is that $\Sigma_g$ would be contained in a particular Noether--Lefschetz divisor in $\mathcal F_g$, which we call the Borisov-Nuer divisor and we denote by $\mathcal{BN}_g$. In this short note, we prove that $\Sigma_g\cap\mathcal{BN}_g$ is non--empty whenever $(g-1)$ is divisible by $4$. To this end, we construct polarized Enriques surfaces $(Y, H_Y)$, with $H_Y^2$ divisible by $4$, which verify the conjecture. 
In particular, the conjecture holds also for any element $\mathcal M_{En,h}^a$, if $h^2$ is divisible by~$4$ and $h$ is the same type of polarization.
\end{abstract}

\maketitle
\section{Introduction}\label{Section:Introduction}

Let $Y$ be an {Enriques surface} over $\mathbb{C}$, that is, a smooth projective surface with $p_g(Y) = q(Y) = 0$ and $2K_Y = \mathcal O_Y$. The universal covering of $Y$ is given by an {\'e}tale double cover map $\sigma_Y : X_Y \to Y$ where $X$ is a $K3$ surface. Hence, an Enriques surface $Y$ determines a pair $(X_Y, \theta_Y)$, where $X_Y$ is its $K3$ cover, and $\theta_Y$ is a fixed-point-free involution on $X_Y$ so that $\sigma_Y$ coincides with the quotient map $X_Y \to X_Y/\theta_Y$. In particular, studying Enriques surfaces $Y$ is equivalent to studying pairs $(X, \theta)$ of $K3$ surfaces $X$ and fixed-point-free involutions $\theta$ on $X$.

A \emph{polarized Enriques surface} is a pair $(Y, H_Y)$, where $Y$ is an Enriques surface and $L \in \pic(Y)$ is an ample line bundle. A \emph{numerically polarized Enriques surface} is a pair $(Y, [H_Y] )$, where $[H_Y] \in \num(Y)\cong U\oplus E_8(-1)$ denotes the numerical class of an ample line bundle $H_Y$ on $Y$. Fix a primitive vector $h\in U\oplus E_8(-1)$. Thanks to the lattice theory, Gritsenko and Hulek were able to give a construction of the moduli space ${\mathcal M}_{En,h}^a$ of numerically polarized Enriques surfaces with a polarization of type $h$ as an open subvariety of a modular variety $\mathcal M_{En,h}$, see \cite{GH16} for details. It is a 10-dimensional quasi-projective variety, and the locus ${\mathcal M}_{En,h}^{nn}$ corresponding to unnodal surfaces (\emph{i.e.}, with no smooth $(-2)$-curves)  is open. For an alternate approach to moduli spaces $\overline{\mathcal E}_{g,\phi}$ using the invariant $\phi$, we refer to \cite{CDGK18}.

Let us consider  $\mathcal F_g$ the moduli space of polarized $K3$ surfaces of genus $g=h^2+1$. Note that $g$ is odd and $g\ge 5$. For any numerical type $h$, we have a natural map
\[
\begin{array} {cccl}
\eta_h: \mathcal M_{En,h}^a & \to & \mathcal F_g\\
  (Y, h=[H_Y]) & \mapsto & (X_Y, \sigma_Y^{*} H_Y). \\
\end{array}
\]
Then the locus
\[
\Sigma_g := \bigcup_{h^2 = g-1} im(\eta_h) \subseteq \sF_g
\]
consists of polarized $K3$ surfaces $(X,H_X)$ which appear as pullbacks of polarized Enriques surfaces $(Y,H_Y)$. Notice that, for any fixed degree $g-1$, there are only finitely many numerical types $h$ with $h^2=g-1$. Indeed, from \cite[Proposition 3.4]{CDGK18} it follows that the number of irreducible components of the moduli space $\overline{\mathcal E}_{g^{\prime},\phi}$ coincides with the number of possible simple decomposition types for $h$ for fixed values $h^2 = 2g^{\prime}-2$ and $\phi(h) = \phi$. Since $0 < \phi^2 \le h^2$ by \cite[Corollary 2.7.1]{CD89}, there are only finitely many possible choices of $\phi$, which implies the claim.

In this note, we discuss a conjecture of Borisov and Nuer on the Enriques lattice $\num(Y) \simeq U \oplus E_8(-1)$, motivated by the Ulrich bundle existence problem, and connect it to the maps $\eta_h$. Let us briefly recall what are Ulrich bundles. Let $X \subset \p^N$ be a smooth projective variety of dimension $n$, and let $H= \o_X(1)$ be a very ample line bundle on $X$. A vector bundle $\mathcal E$ on $X$ which satisfies the following cohomology vanishing condition
\begin{equation}\label{equation:beingUlrich}
H^i (X, \mathcal E(-j)) = 0 \text{ for all } i \text{ and } 1 \le j \le n
\end{equation}
 is called an {Ulrich bundle} on $X$ \cite{ESW03}. They have many interesting applications, in particular, they connect several different topics in algebra and geometry, see \cite{ESW03,Bea18}. One important problem within this topic is to find an Ulrich bundle of smallest possible rank on a given variety. For an Enriques surface $Y$, together with a very ample line bundle $H_Y=\mathcal O_Y(1)$, it is known that $Y$ always carries an Ulrich bundle of rank $2$ \cite{Bea16, Cas17}. On the other hand, Borisov and Nuer observed that the existence of an Ulrich line bundle $N$ on a polarized unnodal Enriques surface $(Y, H_Y)$ is equivalent to the numerical condition
\begin{equation}\label{equation:UlrichOnUnnodalEnriques}
(N-H_Y)^2 = (N-2H_Y)^2 = -2,
\end{equation}
that is, $H_Y$ can be written as a difference of two $(-2)$-line bundles. Here, the unnodal assumption is required only to assure the vanishing of certain cohomology groups. Thus, it is natural to focus only on the equation (\ref{equation:UlrichOnUnnodalEnriques}). They conjectured that it is always possible to find such a line bundle $N$ for any choice of polarization $H_Y$, or even more, for any line bundle:

\begin{conj}[{\cite[Conjecture 2.2]{BN18}}]\label{conjecture:BorisovNuer}
For any line bundle $H$ on an Enriques surface $Y$, there is a line bundle $N \in \pic(Y)$ such that $(N-H)^2 = (N-2H)^2 = -2$.
\end{conj}

Suppose that $(Y, H_Y)$ verifies the Borisov-Nuer conjecture; we have a line bundle $N$ on $Y$ which satisfies the above equation (\ref{equation:UlrichOnUnnodalEnriques}). We translate the conjecture in terms of line bundles on its $K3$ covers by observing the image under $\eta_h$ defined above. Let $\sigma:X \to Y = X/\theta$ be the universal cover,  $H_X := \sigma^{*} H_Y$, and let $M := \sigma^{*} N$. The equation (\ref{equation:UlrichOnUnnodalEnriques}) is equivalent to
\[
\left\{
\begin{array}{rcl}
H_X^2 & = & 2g - 2, \\
M^2 & = & 4g-8, \\
H_X \cdot M ~ & = & 3g-3,
\end{array}
\right.
\]
where $g = H_Y^2 + 1 \ge 5$ is an odd integer. Hence, if $(Y, H_Y)$ and $N$ satisfy the equation (\ref{equation:UlrichOnUnnodalEnriques}), then the image $(X, H_X)$ must lie in the \emph{Borisov-Nuer divisor} $\mathcal {BN}_g$, which is a Noether-Lefschetz divisor $\mathcal {NL}_{2g-3,3g-3} \subset \mathcal F_g$ (the subscript stands for the numbers $\left(\frac{1}{2} M^2 + 1 \right)$ and $H_X \cdot M$, respectively). 

Note that a line bundle $M$ on the $K3$ cover $X$ is contained in $\sigma^{*} \pic(Y)$ if and only if $\theta^{*}M \simeq M$ \cite{Hor78}. In the case, the pushforward $\sigma_{*} M$ splits as a direct sum of two line bundles $N \oplus (N \otimes K_Y)$, where $M \simeq \sigma^{*} N$. We consider the sublocus
\[
\Xi_g = 
\left\{ \begin{array}{l|l}
&  \exists ~ \theta : X \to X \text{ fixed-point-free involution }\\
& \text{ such that } \theta^{*} H_X \simeq H_X, \\
(X, H_X) \in \mathcal F_g & \exists ~ M \in \pic (X) \text{ such that } \theta^{*}M \simeq M,  \\ 
 & \text{ and } (X,H_X,M) \in \mathcal {BN}_g.
\end{array}
\right \}
\]
consisting of polarized $K3$ surfaces of genus $g$ which can be obtained by pullback of some polarized Enriques surface $(Y, H_Y)$ together with a line bundle $N$ so that the triple $(Y, H_Y, N)$ verifies the Conjecture \ref{conjecture:BorisovNuer}. 
In particular, we have $\Xi_g \subseteq \Sigma_g \subset \mathcal F_g$. Since the Picard number $\rho(Y)$ of an Enriques surface $Y$ is $10$, both loci $\Xi_g$ and $\Sigma_g$ have high codimensions in $\mathcal F_g$. With this notation, Conjecture \ref{conjecture:BorisovNuer} implies:
\begin{conj}\label{conj:BorisovNuerAlternativeForm}
The two loci $\Xi_g$ and $\Sigma_g$ coincide.
\end{conj}

Since the locus $\Xi_g$ is contained in the Borisov-Nuer divisor $\mathcal {BN}_{g}$ by definition, this conjecture admits the following much weaker version:

\begin{ques}
Is $\Sigma_g$ contained in $\mathcal{BN}_g$?
\end{ques}

At the moment, the Borisov-Nuer conjecture is known for only a few examples: Fano polarization $\Delta$ and its multiple $k \Delta$ by Borisov and Nuer themselves \cite[Theorem 2.4]{BN18}, and a degree $4$ polarization \cite[Theorem 13]{AK17}. In particular, $\Xi_g$ is nonempty when $g=5$ or $g=11$. To have a better understanding, it is worthwhile to observe $\Xi_g$, and to collect more evidences for the Borisov-Nuer conjecture.

In this paper, we construct  examples of points in $\Xi_g$ for various values of $g$.
Suppose that Conjecture \ref{conjecture:BorisovNuer} holds for a numerically polarized Enriques surface $(Y, [H_Y])$ with $H_Y$ of type $h$. Since all the Enriques surfaces $Y$ have the same lattice structure $\num (Y) \simeq U \oplus E_8(-1)$, we immediately have that Conjecture \ref{conjecture:BorisovNuer} holds for every numerically polarized Enriques surface $(Y^{\prime}, [H_Y^{\prime}]) \in {\mathcal M}_{En,h}^a$. Hence, it suffices to construct only one numerically polarized Enriques surface $(Y, h)$ from the moduli space $\mathcal M_{En,h}^a$ which makes Conjecture \ref{conjecture:BorisovNuer} hold. The key ingredient is a Jacobian Kummer surface $X=Km(C)$ of a general curve $C$ of genus $2$, similar as in \cite{AK17}. Such a Jacobian Kummer surface has plenty of technical merits, for instance:

\begin{itemize}
\item $X$ has a fixed-point-free involution $\theta$, that is, $X$ is the $K3$ cover of some Enriques surface $Y$; 
\item intersection theory of $X$ is well-understood;
\item the pullback homomorphism $\theta^{*} : \pic (X) \to \pic(X)$ is well-understood;
\item the Picard number $\rho(X)$ is quite big, so there are more chances to find a certain line bundle.
\end{itemize}

The main result of this paper is the nonemptyness of the locus $\Xi_g$ for various values $g$ as follows, see Theorem \ref{thm:Xi_g is nonempty}:

\begin{thm*}
When $g-1$ is divisible by $4$, the locus $\Xi_g$ is nonempty. In other words, for any given $k>0$ and any Enriques surface $Y$, there is an ample and globally generated line bundle $H_Y$ and a line bundle $N$ on $Y$ such that $H_Y^2 = 4k$ and $(N-H_Y)^2 = (N-2H_Y)^2 = -2$.
\end{thm*}

The outline of the paper is the following. In Section \ref{Section:Preliminaries}, we review some basic facts on Enriques surfaces, Jacobian Kummer surfaces as $K3$ covers of Enriques surfaces, and line bundles. We also fix the notation we use. In Section \ref{Section:K3Cover}, we describe a construction of a polarized Enriques surface which verifies the Borisov-Nuer conjecture using a Jacobian Kummer surface and we provide a few more examples in the case when $(g-1)$ is not divisible by $4$.

\section{Preliminaries}\label{Section:Preliminaries}

We recall some basic facts on Enriques surfaces and Jacobian Kummer surfaces. As the above discussion indicates, we translate the Borisov-Nuer conjecture and the equation (\ref{equation:UlrichOnUnnodalEnriques}) on an Enriques surface $Y$ in terms of line bundles on its $K3$ cover $X$. To construct an Enriques surface from its $K3$ cover, we need a $K3$ surface $X$ together with a fixed-point-free involution $\theta$ so that the quotient $X/\theta$ becomes an Enriques surface. Thanks to the following theorem of Keum, we pick algebraic Kummer surfaces as candidates:

\begin{lem}[{\cite[Theorem 2]{Keu90}}]
An algebraic Kummer surface is a K3 cover of an Enriques surface.
\end{lem}

When the covering map $\sigma : X \to X/\theta = Y$ of an Enriques surface is fixed, we also need to ask which line bundles on $X$ are pullbacks of some line bundles on $Y$. The answer is also well-known, thanks to Horikawa.

\begin{lem}[{\cite[Theorem 5.1]{Hor78}}]\label{Lemma:Horikawa}
Let $X$ be a $K3$ surface, $\theta : X \to X$ be a fixed-point-free involution, and $\sigma : X \to Y = X/\theta$ be the $2:1$ {\'e}tale cover. Then the image of the map $\sigma^{*} : \pic (Y) \to \pic (X)$ is the set of line bundles $M$ in $X$ such that $\theta^{*}M \simeq M$.
\end{lem}

Next, we recall the construction of a Jacobian Kummer surface and intersection theory over it. Let $C$ be a generic curve of genus $2$. Its Jacobian variety $\mathcal A = J(C)$ is an Abelian surface with N{\'e}ron-Severi group $NS(\mathcal A) = \mathbb{Z} \cdot [\Theta]$ with $\Theta^2 = 2$. Note that $\mathcal A$ has a natural involution $\iota$ with $16$ fixed point. The complete linear system $|2 \Theta|$ defines a morphism to $\mathbb{P}^3$, which factors through the singular quartic $\mathcal A / \iota$ (Kummer quartic) with $16$ ordinary double points. The Kummer surface $X = Km(\mathcal A)$ is defined as the minimal desingularization of $\mathcal A/\iota$. Throughout the rest of the paper, we fix the notations as follows.

\begin{notation}
We follow the notation as in \cite{AK17}.
\begin{itemize}
\item $C$ : a generic curve of genus 2 with 6 Weierstrass points $p_1, \cdots, p_6 \in C$;
\item $X = Km (C)$ : Jacobian Kummer surface associated to $C$, which is the minimal desingularization of $J(C)/\iota$;
\item $\theta : X \to X$ : a fixed-point-free involution so called ``switch'' induced by the even theta characteristic $[p_4+p_5-p_6]$;
\item $\sigma : X \to Y = X/\theta$ : the quotient map so that $Y$ is an Enriques surface;
\item $L$ : the line bundle induced by the hyperplane section of the singular quartic $J(C)/\iota \subseteq \mathbb{P}^3$;
\item $E_0, E_{ij} ~ (1 \le i < j \le 6)$ : sixteen $(-2)$-curves called \emph{nodes};
\item $T_i ~ (1 \le i \le 6 ), T_{ij6} ~ (1 \le i < j \le 5)$ : sixteen $(-2)$-curves called \emph{tropes}.
\end{itemize}
\end{notation}

\noindent Note that $L^2 = 4, L \cdot E_0 = L \cdot E_{ij} = 0$, and two distinct nodes do not intersect. 

Let us describe the nodes and the tropes more precisely. Following the notation in \cite{Oha09}, the 16 nodes are labeled by the corresponding 2--torsion points in the Jacobian $\mathcal A=J(C)$: 
\begin{eqnarray*}
E_0 & = & \text{ node corresponding to } [0] \in \mathcal{A}; \\
E_{ij}=E_{[p_i - p_j]} & = & \text{ node corresponding to } [p_i - p_j] \in \mathcal{A}, 1 \le i < j \le 6.
\end{eqnarray*}
The tropes are labeled using their associated theta--characteristics of $C$ \cite{Oha09}, e.g. $T_i=T_{[p_i]}$ corresponds to $[p_i]$ and $T_{ijk}=T_{[p_i+p_j-p_k]}$ corresponds to $[p_i+p_j-p_k]$ for any $i<j<k$. Note that $T_{ijk}=T_{\ell m n}$ if $\{i,j,k\}\cup \{\ell, m,n\}=\{1,\ldots,6\}$.

Also note that the pullback $\theta^{*}$ swaps the nodes $E_{\alpha}$ and the tropes $T_{\alpha+\beta}$ in the following way, cf. \cite{Muk12} and \cite[Section 4, Section 5]{Oha09}:

\[
\begin{array}{|ccc||ccc|}
\hline
\text{Nodes} & \ & \text{Tropes} &  \text{Nodes} & \ & \text{Tropes} \\
\hline
E_0 & \leftrightarrow & T_{456}  & E_{25} & \leftrightarrow & T_{246} \\
E_{12} & \leftrightarrow & T_{3} &  E_{26} & \leftrightarrow & T_{136} \\
E_{13} & \leftrightarrow & T_{2} & E_{34} & \leftrightarrow & T_{356} \\
E_{14} & \leftrightarrow & T_{156} & E_{35} & \leftrightarrow & T_{346} \\
E_{15} & \leftrightarrow & T_{146} & E_{36} & \leftrightarrow & T_{126} \\
E_{16} & \leftrightarrow & T_{236} & E_{45} & \leftrightarrow & T_{6} \\
E_{23} & \leftrightarrow & T_{1} & E_{46} & \leftrightarrow & T_{5} \\
E_{24} & \leftrightarrow & T_{256} & E_{56} & \leftrightarrow & T_{4} \\
\hline
\end{array}
\]
\\
where the corresponding tropes are 
\begin{eqnarray*}
T_i & = & \frac{1}{2} (L - E_0 - \sum_{k \neq i} E_{ik})
\end{eqnarray*}
for $1\le i\le 6$ and 
\begin{eqnarray*}
T_{ij6} & = & \frac{1}{2} (L - E_{i6} - E_{j6} - E_{ij} - E_{\ell m}-E_{mn}-E_{\ell n})
\end{eqnarray*}
for $1 \le i < j \le 5$, where $\{l,m,n \}$ is the complement of $\{i, j\}$ in $\{1, 2, 3, 4, 5\}$ (see \cite[Lemma 4.1]{Oha09}).

It is well-known that $\{ E_0, E_{ij}, T_i, T_{ij6} \}$ spans $\pic(X)$ \cite[Lemma 3.1]{Keu97}, and hence $\{ L, E_0, E_{ij} \}$ spans $\pic(X) \otimes \frac{1}{2} \mathbb{Z}$ if we allow $\frac{1}{2} \mathbb{Z}$ coefficients. For simplicity, we mostly consider a linear combination of $L, E_0, E_{ij}$ in $\frac{1}{2} \mathbb{Z}$ coefficients, however, we have to carefully choose the coefficients so that the linear combination gives an element in $\pic(X)$.

\section{Construction using K3 covers}\label{Section:K3Cover}

Let $(Y, H_Y)$ be a polarized Enriques surface, and let $\sigma : X \to Y=X/\theta$ be its $K3$ cover. Suppose it verifies Conjecture \ref{equation:UlrichOnUnnodalEnriques}, that is, $Y$ has a line bundle $N$ which fits into the equation (\ref{equation:UlrichOnUnnodalEnriques}). The equation (\ref{equation:UlrichOnUnnodalEnriques}) can be completely translated into the numerical conditions on its $K3$ cover. Namely, we are interested in line bundles $M \in \sigma^{*} \pic(Y) \subseteq \pic(X)$ which verifies the equation 
\begin{equation}\label{equation:NumericalConditionOnK3}
(M-H_X)^2 = (M-2H_X)^2 = -4
\end{equation}
where $H_X := \sigma^{*} H_Y$. Note that if $H_Y$ is ample and globally generated, then $H_X$ is also ample and globally generated, and vice versa.

Now let $X$ be a Jacobian Kummer surface associated to a generic curve $C$ of genus $2$. As mentioned in the previous section, some line bundles in $\pic(X)$ require  rational coefficients in $\frac{1}{2}\mathbb{Z}$ when we write it as linear combinations of $L$ and nodes $E_{ij}$. One typical example is called an even eight:

\begin{lem}
The set of $8$ nodes $\{ E_0, E_{16}, E_{23}, E_{24}, E_{25}, E_{34}, E_{35}, E_{45} \}$ forms an even eight, that is, $\left( E_0 + E_{16} + E_{23} + E_{24} + E_{25} + E_{34} + E_{35} + E_{45} \right)$ is divisible by $2$ in $\pic(X)$.
\end{lem}

\begin{proof}
It is straightforward from a direct computation
\begin{eqnarray*}
L - T_4 -  E_{14} - T_{146} - E_{46} & = & L - \frac{1}{2} \left( L - E_0 - E_{14} - E_{24} - E_{34} - E_{45} - E_{46} \right) - E_{14} \\
& & - \frac{1}{2} \left(L - E_{14} - E_{16} - E_{46} - E_{23} - E_{25} - E_{35} \right)  - E_{46} \\
& = & \frac{1}{2} \left( E_0 + E_{16} + E_{23} + E_{24} + E_{25} + E_{34} + E_{35} + E_{45} \right).
\end{eqnarray*}
\end{proof}

Also note that the complementary set of nodes $\{ E_{12}, E_{13}, E_{14}, E_{15}, E_{26}, E_{36}, E_{46}, E_{56}\}$ also forms an even eight. Since
\begin{eqnarray*}
\theta^{*} (E_{12}+E_{15}+E_{26}+E_{56}) & = & T_3 + T_{146} + T_{136} + T_4 \\
& = & 2L - \sum_{i,j}^{16} E_{ij} + (E_{12} + E_{15} + E_{26} + E_{56})
\end{eqnarray*}
and by similar computations, grouping them by those 4 line bundles makes the problem easier. Let $F_{\bullet}$ be the sum of four nodes $E_{ij}$, namely,
\[\left\{
\begin{array}{ccccccccc}
F_1 &= & E_{12} &+&E_{15}&+& E_{26}&+&E_{56} \\
F_2 & = & E_{13} &+ &E_{14}&+& E_{36}&+&E_{46} \\
F_3 & = & E_{23} &+&E_{25}&+&E_{34}&+&E_{45} \\
F_4 & = & E_{0}  &+&E_{16}&+&E_{24}&+&E_{35} \\
\end{array}
\right.
\]
We have

\begin{eqnarray*}
\theta^{*}L & \simeq & 3L - \sum_{i,j}^{16} E_{ij} \\
\theta^{*}F_k & \simeq & 2L - \sum_{i,j}^{16} E_{ij} + F_k \text{ for each } 1 \le k \le 4.
\end{eqnarray*}

Consider a linear combination of the form $M = \alpha L - \beta_1 F_1 - \beta_2 F_2 - \beta_3 F_3 - \beta_4 F_4$ as a special case. First, we need to check when $M$ becomes a $\theta^{*}$-invariant line bundle on $X$.

\begin{lem}\label{Lemma:ConditionThetaInvariantLB}
A linear combination $M = \alpha L - \beta_1 F_1 - \beta_2 F_2 - \beta_3 F_3 - \beta_4 F_4$ is a line bundle in $\pic(X)$ such that $\theta^{*} M \simeq M$ if and only if $\beta_i \in \frac{1}{2} \mathbb{Z}$, $\beta_1 + \beta_2 \in \mathbb{Z}$, $\beta_3 + \beta_4 \in \mathbb{Z}$, and $\alpha = \beta_1 + \beta_2 + \beta_3 + \beta_4$.
\end{lem}

\begin{proof}
Recall that $\pic(X)$ is spanned by integral linear combinations of nodes $E_{ij}$ and tropes $T_{i}, T_{ij6}$. In particular, $\alpha, \beta_i \in \frac{1}{2} \mathbb{Z}$. We first check the condition $\theta^{*} M \simeq M$. A direct computation shows that $\theta^{*} M \simeq M$ if and only if $\alpha = \beta_1 + \beta_2 + \beta_3 + \beta_4$.

We still need to show that $M \in \pic(X)$.
Since $F_1 + F_2$ and $F_3 + F_4$ are divisible by 2 in $\pic(X)$, but no other $F_i + F_j$ are divisible by $2$ \cite[Proposition V.6]{Meh06} , hence the coefficients $\beta_i$ are elements in $\frac{1}{2} \mathbb{Z}$ such that $\beta_1+\beta_2 \in \mathbb{Z}$ and $\beta_3 + \beta_4 \in \mathbb{Z}$. 
\end{proof}

\begin{exmp}
Let $\beta_1 = \beta_2 = \beta_3 = \beta_4 = \frac{1}{2}$, and $\alpha = \sum \beta_i = 2$. The line bundle $H_X = 2L - \frac{1}{2} (F_1 + F_2 + F_3 + F_4)$ satisfies the assumptions in Lemma \ref{Lemma:ConditionThetaInvariantLB}, and defines an embedding of $X$ into $\mathbb{P}^5$ as the intersection of $3$ quadrics \cite[Theorem 2.5]{Shi77}. Such a Kummer surface $(X, H_X)$ carries a line bundle $M$ such that $\theta^{*} M \simeq M$, namely,
\[
M = 3L - F_1 - F_2 - F_4
\]
as in \cite[proof of Theorem 13]{AK17}. Furthermore, $H_X$ and $M$ satisfies the equation (\ref{equation:NumericalConditionOnK3}) as desired.
\end{exmp}

Let $H_X = \alpha L - \sum_{k=1}^4 \beta_k F_k$, and let $M = \alpha^{\prime} L - \sum_{k=1}^4 \beta_k^{\prime} F_k$. Suppose that both  $H_X$ and $M$ satisfies the assumptions in Lemma \ref{Lemma:ConditionThetaInvariantLB}. Now our question becomes: 

\begin{ques}
For a given ample polarization $H_X = \alpha L - \sum_{k=1}^4 \beta_k F_k$, find values $\beta_i^{\prime}$ so that the line bundles $H_X$ and $M$ verify the equation (\ref{equation:NumericalConditionOnK3}).
\end{ques}

By taking the substitutions 
\[
\left\{
\begin{array}{ccc}
S & = & \beta_1^{\prime} - \beta_1 \\
T & = & \beta_2^{\prime} - \beta_2 \\
U & = & \beta_3^{\prime} - \beta_3 \\
V & = & \beta_4^{\prime} - \beta_4 \\
\end{array}
\right. ,
\]
the equation (\ref{equation:NumericalConditionOnK3}) gives the system of two quadratic Diophantine equations, namely:
\begin{eqnarray*}
4(S+T+U+V)^2 - 8S^2 - 8T^2 - 8U^2 - 8V^2  & = & -4, \\
4(\alpha - (S+T+U+V))^2 - 8(\beta_1 - S)^2 - 8(\beta_2 - T)^2 &&\\
- 8(\beta_3 - U)^2 - 8(\beta_4 - V)^2  & = & -4.
\end{eqnarray*}

Dividing both equations by 4 and taking their difference, we have
\begin{eqnarray}
(S+T+U+V)^2 - 2S^2 - 2T^2 - 2U^2 - 2V^2 & = & -1 \label{Equation:DiophantineQuadratic}\\
2 \alpha (S+T+U+V) - 4 \beta_1 S - 4 \beta_2 T - 4 \beta_3 U - 4 \beta_4 V - \frac{d}{4} & = & \phantom{-}0 \label{Equation:DiophantineLinear}
\end{eqnarray}
where $\alpha = \beta_1 + \beta_2 + \beta_3 + \beta_4$ and $d = H_X^2 = 4 \alpha^2 - 8(\beta_1^2 + \beta_2^2 + \beta_3^2 + \beta_4^2)$. Therefore, finding $M$ is equivalent to finding a solution $(S,T,U,V)$ of this system of Diophantine equations (\ref{Equation:DiophantineQuadratic}), (\ref{Equation:DiophantineLinear}), where the corresponding $M$ satisfies the assumptions in Lemma \ref{Lemma:ConditionThetaInvariantLB}.

In most cases, finding integral solutions of a system of Diophantine equations is extremely hard even though it has rationally parametrized solutions. Instead, we provide a sufficient condition on $\beta_i$'s so that the system has a solution $(S,T,U,V)$ which fits into all the conditions we need.

\begin{prop}\label{Prop:SuffConditionDiophantineEquationOnK3}
Let $\beta_1, \beta_2, \beta_3, \beta_4 \in \frac{1}{2} \mathbb{Z}$ such that $\beta_1 + \beta_2 \in \mathbb{Z}$, $\beta_3 + \beta_4 \in \mathbb{Z}$, and 
\[
2S = \frac{1}{2 (\beta_3 + \beta_4)} \left[ (\beta_1 + \beta_2 + \beta_3 + \beta_4)^2 - 2 (\beta_1^2 + \beta_2^2 + \beta_3^2 + \beta_4^2 ) + 2 (\beta_3 - \beta_4) \right] \in \mathbb{Z}.
\]
Then the above system of Diophantine equations has a solution $(S,T,U,V) = (S,S,\frac{1}{2}, - \frac{1}{2})$ so that $\beta_1^{\prime}, \cdots, \beta_4^{\prime}$ satisfy the assumptions in Lemma \ref{Lemma:ConditionThetaInvariantLB}. 
\end{prop}

\begin{proof}
It is clear that $(S,S,\frac{1}{2}, - \frac{1}{2})$ is a solution for the equation (\ref{Equation:DiophantineQuadratic}). Substitute into the equation (\ref{Equation:DiophantineLinear}), we have a univariable linear equation
\[
4(\beta_3 + \beta_4) S - (\beta_1 + \beta_2 + \beta_3 + \beta_4)^2 + 2 (\beta_1^2 + \beta_2^2 + \beta_3^2 + \beta_4^2) - 2 \beta_3 + 2 \beta_4 = 0.
\]
It is straightforward that such a solution $(S,T,U,V) = (S,S,\frac{1}{2},-\frac{1}{2})$ provides $\beta_1^{\prime}, \beta_2^{\prime}, \beta_3^{\prime}, \beta_4^{\prime}$ which satisfies the assumptions in Lemma \ref{Lemma:ConditionThetaInvariantLB}.
\end{proof}

By taking suitable quadruples $(\beta_1, \beta_2, \beta_3, \beta_4)$, we obtain a number of polarized Enriques surfaces establishing the Borisov-Nuer conjecture as follows.

\begin{prop}\label{Proposition:Main}
Suppose that $H_X = (\beta_1 + \beta_2 + \beta_3 + \beta_4)L - \beta_1 F_1 - \beta_2 F_2 - \beta_3 F_3 - \beta_4 F_4$ is an ample and globally generated line bundle on $X$ such that 
$\beta_1, \cdots, \beta_4$ satisfy the assumptions in Proposition \ref{Prop:SuffConditionDiophantineEquationOnK3}. Then there is a polarized Enriques surface $(Y, H_Y)$ and a line bundle $N$ on $Y$ such that $H_X^2 = 2 H_Y^2$ and $(N-H_Y)^2 = (N - 2H_Y)^2 = -2$. In particular, the Borisov-Nuer conjecture holds for $(Y, H_Y)$.

\end{prop}

\begin{proof}
Let $S = \frac{1}{4 (\beta_3 + \beta_4)} \left[ (\beta_1 + \beta_2 + \beta_3 + \beta_4)^2 - 2 (\beta_1^2 + \beta_2^2 + \beta_3^2 + \beta_4^2 ) + 2 (\beta_3 - \beta_4) \right]$. Proposition \ref{Prop:SuffConditionDiophantineEquationOnK3} implies that $(S,S, \frac{1}{2}, -\frac{1}{2})$ is a solution of the system of Diophantine equations (\ref{Equation:DiophantineQuadratic}), (\ref{Equation:DiophantineLinear}). Hence, the line bundle 
\[
M := (\beta_1 + \beta_2 + \beta_3 + \beta_4 + 2S) L - (\beta_1 + S) F_1 - (\beta_2 + S) F_2 + \left(\beta_3 + \frac{1}{2} \right) F_3 - \left( \beta_4 - \frac{1}{2} \right) F_4
\]
verifies the conditions $\theta^{*}M \simeq M$ and $(M-H_X)^2 = (2H_X - M)^2 = -4$.

By Lemma \ref{Lemma:Horikawa}, there are line bundles $H_Y$ and $N$ on an Enriques surface $Y = X/\theta$ such that $\sigma^* H_Y = H_X$, $\sigma^* N = M$ where $\sigma : X \to Y=X/\theta$ is the quotient map. Since $\sigma_{*} H_X = H_Y \oplus (H_Y \otimes K_Y)$ and $\sigma_{*} M = N \oplus (N \otimes K_Y)$, we conclude that $(N-H_Y)^2 = (2H_Y - N)^2 = -2$. 
\end{proof}

Together with a discussion on the moduli of (numerically) polarized Enriques surfaces, we get the following non-emptiness.
\begin{thm}\label{thm:Xi_g is nonempty}
The locus $\Xi_g$ contained in the Borisov-Nuer divisor $\mathcal {BN}_g \subset \mathcal F_g$ of polarized $K3$ surfaces of degree $H_X^2 = 2g-2$ is nonempty when $2g-2$ is divisible by $8$. In particular, there is a numerically polarized Enriques surface $(Y, h=[H_Y]) \in \mathcal{M}_{En,h}^a$ which verifies the Borisov-Nuer conjecture when $h^2 = g-1$ is divisible by $4$. Moreover, the conjecture also holds for every $(Y^{\prime}, [H_{Y^{\prime}}]) \in \mathcal M_{En,h}^a$. 
\end{thm}

\begin{proof}
Let $X$ be a general Jacobian Kummer surface as above. It suffices to construct a pair $(H_X,M)$ of line bundles on $X$ determined by the values $\beta_i$'s and $\beta_i^{\prime}$'s satisfying Proposition \ref{Prop:SuffConditionDiophantineEquationOnK3}. Suppose $g=4k+1$ so that $H_X^2 = 8k$ is divisible by 8. We pick $H_X = (k+1)L - \frac{k}{2} (F_1 + F_2) - \frac{1}{2} (F_3 + F_4)$ so that $H_X^2 = 8k, \beta_1 = \beta_2 = \frac{k}{2}, \beta_3=\beta_4 = \frac{1}{2}$. Note that $H_X = \left[ 2L - \frac{1}{2} (F_1 + F_2 + F_3 + F_4) \right] +  (k-1) \left[ L - \frac{1}{2} (F_1 + F_2) \right]$ is a sum of two line bundles. Since the former one is very ample, and the later one is a multiple of a line bundle which induces an elliptic fibration over $\p^1$ (see \cite[Fibration 7]{Kum14} and \cite[Section 5.1]{GS16}), their sum $H_X$ is indeed ample and globally generated.

Moreover, the value 
\[
\frac{1}{2 (\beta_3 + \beta_4 ) } \left[ (\beta_1 + \beta_2 + \beta_3 + \beta_4)^2 - 2(\beta_1^2 + \beta_2^2 + \beta_3^2 + \beta_4^2) + 2 (\beta_3 - \beta_4) \right] = k
\]
is an integer, we conclude that there is a line bundle $M$ which verifies the equation
\[
(M-H_X)^2 = (M-2H_X)^2 = -4
\]
by Proposition \ref{Prop:SuffConditionDiophantineEquationOnK3}. For instance, we may take $M = (2k+1)L- k(F_1 + F_2) -  F_3$. 
\end{proof}

\begin{cor}
Let $(Y, h=[H_Y]) \in \mathcal M_{En,h}^{a}$ be a numerically polarized Enriques surface appearing in Theorem \ref{thm:Xi_g is nonempty}. Let $(Y^{\prime}, [H_Y^{\prime}]) \in \mathcal M_{En,h}^a$ be a generic element. Then the Enriques surface $Y^{\prime}$ has an $H_Y^{\prime}$-Ulrich line bundle, in the sense of \cite[Definition 1]{AK17}.
\end{cor}
\begin{proof}
Note that any $(Y^{\prime}, [H_Y^{\prime}]) \in \mathcal M_{En,h}^a$  carries a line bundle $N^{\prime} \in \pic(Y^{\prime})$ such that $(N^{\prime} - H_Y^{\prime})^2 = (N^{\prime} - 2H_Y^{\prime})^2 = -2$. Since $Y^{\prime}$ is general, it is unnodal; it does not contain any smooth $(-2)$-curves. By \cite[Proposition 2.1]{BN18}, $N^{\prime}$ is an $H_Y^{\prime}$-Ulrich line bundle as desired.
\end{proof}

\begin{exmp}
There are several possible choices of $H_X$ satisfying the assumptions of Proposition \ref{Prop:SuffConditionDiophantineEquationOnK3} and Theorem \ref{thm:Xi_g is nonempty} when we fix the degree $H_X^2$. For instance, take $\beta_1 = \beta_2 = \frac{m}{2}$, $\beta_3 = \beta_4 = \frac{n}{2}$ where $m, n$ are positive integers. The line bundle $H_X := (m+n)L - \frac{m}{2}(F_1 + F_2) - \frac{n}{2} (F_3+F_4)$ is ample and globally generated with the self-intersection number $H_X^2 = 8mn$. Furthermore, the value
\[
\frac{1}{2 (\beta_3 + \beta_4 ) } \left[ (\beta_1 + \beta_2 + \beta_3 + \beta_4)^2 - 2(\beta_1^2 + \beta_2^2 + \beta_3^2 + \beta_4^2) + 2 (\beta_3 - \beta_4) \right] = m
\]
is always an integer, so we are able to find a solution of Diophantine equations (\ref{Equation:DiophantineQuadratic}), (\ref{Equation:DiophantineLinear}).
\end{exmp}

\begin{rem}\label{Remark:DegreeNotDivisibleBy4}
The system of Diophantine equations (\ref{Equation:DiophantineQuadratic}), (\ref{Equation:DiophantineLinear}) needs not to have a desired solution. For example, let $\beta_1 = 1$, $\beta_2 = \beta_3 = \beta_4 = 0$. Then the second equation (\ref{Equation:DiophantineLinear}) becomes
\[
-2S + 2T + 2U + 2V = -1.
\]
Since $-S+T = (\beta_2^{\prime} - \beta_1^{\prime}) - (\beta_2 - \beta_1)$ and $U+V$ are  integers, the left-hand side must be an even integer. Hence, there is no solution which satisfies the assumptions. In general, by a simple parity argument, one can easily check that the system (\ref{Equation:DiophantineQuadratic}), (\ref{Equation:DiophantineLinear}) does not have a solution $(S,T,U,V)$ such that the corresponding $M$ satisfies the assumptions of Lemma \ref{Lemma:ConditionThetaInvariantLB} when the number $\frac{d}{4}$ (which stands for $\frac{1}{4} H_X^2$ in the context) is not an even integer. This is the reason why it is not easy to verify the nonemptiness of $\Xi_g$ when $g-1$ is not divisible by $4$. For instance, we cannot verify that Borisov-Nuer conjecture holds for a Fano polarized Enriques surface $(Y, \Delta)$ in the above arguments, since $g-1 = \Delta^2 = 10$ is not divisible by $4$.

However, there might be plenty of chances to find a solution of the equation (\ref{equation:NumericalConditionOnK3}) using the same Jacobian Kummer surface. We only address a few more examples as evidence. We cannot guarantee that the following bundles $H_X$ are ample and/or globally generated, however, this aspect is not very important from the viewpoint of the original Borisov-Nuer conjecture.  

\begin{enumerate}[(i)]
\item Let $H_X = 4L - 2F_1 - F_2 - \frac{1}{2}F_3 - \frac{1}{2}F_4$ so that $\theta^{*}H_X \simeq H_X$ and $H_X^2 = 20$. We take $M$ as
\[
M = 6L - 3F_1 - \frac{3}{2} (E_0+E_{13}+E_{14}+E_{16}+E_{25}+E_{34}+E_{36}+E_{46}).
\] 
Since
\[
L-T_1-T_{346}+E_{12}+E_{15} = \frac{1}{2} (E_0+E_{13}+E_{14}+E_{16}+E_{25}+E_{34}+E_{36}+E_{46}),
\]
$M$ is a line bundle on $X$. Furthermore, $M$ satisfies $\theta^{*}M \simeq M$ and $(M-H_X)^2 = (M-2H_X)^2 = -4$. Hence, there is an Enriques surface $Y$ and two line bundles $H_Y, N$ with $H_Y^2=10$ such that $(N-H_Y)^2 = (2H_Y-N)^2 = -2$. 

\item Let $H_X = 6L - 3F_1 - 2F_2 - \frac{1}{2}F_3 - \frac{1}{2}F_4$ so that $\theta^{*}H_X \simeq H_X$ and $H_X^2 = 36$. We take $M$ as
\[
M = 8L - \frac{7}{2}F_1 - \frac{3}{2}F_2 - \frac{3}{2}(E_0 + E_{13}+E_{14}+E_{16}+E_{25}+E_{34}+E_{36}+E_{46}).
\]
We have $M \in \pic(X)$, $\theta^{*}M \simeq M$, and $H_X, M$ satisfy the equation (\ref{equation:NumericalConditionOnK3}).

\item Let $H_X = 8L - 4F_1 - 3F_2 - \frac{1}{2}F_3 - \frac{1}{2}F_4$. We have $\theta^{*}H_X \simeq H_X$ and $H_X^2 = 52$. We take $M$ as \[
M = 10L - 4(F_1+F_2) - \frac{1}{2}(E_0 + E_{13}+E_{14}+E_{16}+E_{25}+E_{34}+E_{36}+E_{46}) - (E_{23}+E_{24}+E_{35}+E_{45}).
\]
We have $M \in \pic(X)$, $\theta^{*}M \simeq M$, and $H_X, M$ satisfy the equation (\ref{equation:NumericalConditionOnK3}).
\end{enumerate}
\end{rem}

\begin{ack}
MA was partly supported by a grant of Ministery of Research and Innovation, CNCS--UEFISCDI, project number PN-III-P4-ID-PCE-2016-0030, within PNCDI III. YK thanks Simion Stoilow Institute of Mathematics of the Romanian Academy (IMAR) for the hospitality during his visit. He was partly supported by Project I.6 of SFB-TRR 195 ``Symbolic Tools in Mathematics and their Application'' of the German Research Foundation (DFG). 
\end{ack}

\def\cprime{$'$} \def\cprime{$'$} \def\cprime{$'$} \def\cprime{$'$}
  \def\cprime{$'$} \def\cprime{$'$} \def\dbar{\leavevmode\hbox to
  0pt{\hskip.2ex \accent"16\hss}d} \def\cprime{$'$} \def\cprime{$'$}
  \def\polhk#1{\setbox0=\hbox{#1}{\ooalign{\hidewidth
  \lower1.5ex\hbox{`}\hidewidth\crcr\unhbox0}}} \def\cprime{$'$}
  \def\cprime{$'$} \def\cprime{$'$} \def\cprime{$'$}
  \def\polhk#1{\setbox0=\hbox{#1}{\ooalign{\hidewidth
  \lower1.5ex\hbox{`}\hidewidth\crcr\unhbox0}}} \def\cdprime{$''$}
  \def\cprime{$'$} \def\cprime{$'$} \def\cprime{$'$} \def\cprime{$'$}
\providecommand{\bysame}{\leavevmode\hbox to3em{\hrulefill}\thinspace}
\providecommand{\MR}{\relax\ifhmode\unskip\space\fi MR }
\providecommand{\MRhref}[2]{%
  \href{http://www.ams.org/mathscinet-getitem?mr=#1}{#2}
}
\providecommand{\href}[2]{#2}


\end{document}